\documentclass[a4paper,10pt]{amsart}
\usepackage{amsmath,amssymb,amsthm,graphicx}
\title{Triangles with two given integral sides}
\author{Sz. Tengely}
\thanks{Research supported in part by the Magyary Zolt\'an Higher Educational Public Foundation}
\address{Mathematical Institute\newline
 \indent University of Debrecen\newline
 \indent P.O.Box 12\newline
 \indent 4010 Debrecen\newline
 \indent Hungary}
\email{tengely@math.klte.hu}

\begin{document}
\newcommand{\Bound}{\mbox{Bound}}
\newcommand{\NE}{\mbox{NumofEq}}
\newcommand{\Root}{\mbox{Root}}
\newcommand{\pol}{\mbox{pol}}
\newcommand{\Red}{\mbox{Reduction}}
\newcommand{\ord}{\mbox{\rm ord}}
\newcommand{\lcm}{\mbox{\rm lcm}}
\newcommand{\sign}{\mbox{\rm sign}}
\newcommand{\ggd}{\mbox{\rm ggd}}

\newtheorem{thm}{Theorem}
\newtheorem{lem}{Lemma}
\newtheorem*{cor}{Corollary}
\newtheorem*{thm1}{Theorem}
\newtheorem*{lem1}{Lemma}
\newtheorem*{conj1}{Conjecture}
\theoremstyle{definition}
\newtheorem*{rem}{Remark}
\newtheorem*{acknowledgement}{Acknowledgement}
\bibliographystyle{plain}
\maketitle

\begin{abstract}
We study some Diophantine problems related to triangles with two given integral sides. We solve two problems posed by Zolt\'an Bertalan and we also provide some generalization.
\end{abstract}

\section{introduction}
There are many Diophantine problems arising from studying certain properties of triangles. Most people know the theorem on the lengths of sides of right angled triangles named after Pythagoras. That is $a^2+b^2=c^2.$ 

An integer $n\geq 1$ is called congruent if it is the area of a right triangle
with rational sides. Using tools from modern arithmetic theory of elliptic curves and modular forms 
Tunnell \cite{Tunnell} found necessary condition for $n$ to be a congruent number.
Suppose that $n$ is a square–free positive integer which is a congruent number. 
\begin{itemize}
\item[(a)] If $n$ is odd, then the number of integer triples $(x, y, z)$ satisfying the equation
$n = 2x^2 + y^2 + 8z^2$ is just twice the number of integer triples $(x, y, z)$ satisfying 
$n =2x^2 + y^2 + 32z^2.$
\item[(b)]  If $n$ is even, then the number of integer triples $(x, y, z)$ satisfying the equation
$\frac{n}{2}= 4x^2 + y^2 + 8z^2$  is just twice the number of integer triples $(x, y, z)$ satisfying
$\frac{n}{2}=4x^2 + y^2 + 32z^2.$
\end{itemize}

A Heronian triangle is a triangle having the property that the lengths of its sides and its area are positive integers. There are several open problems concerning the existence of Heronian
triangles with certain properties. 
It is not known whether there exist Heronian triangles having the property that the lengths of
all their medians are positive integers \cite{Guy1994}, and it is not known whether there exist Heronian triangles having the property that the lengths of all their sides are Fibonacci numbers \cite{HaKeRo}.
Ga\'al, J\'ar\'asi and Luca \cite{GaJaLu} proved that there are only finitely many Heronian triangles whose sides $a,b,c\in S$ and are reduced, that is $\gcd(a,b,c)=1,$ where $S$ denotes the set of integers divisible only by some fixed primes.

Petulante and Kaja \cite{PeKa} gave arguments for parametrizing all integer-sided triangles that contain a specified angle with rational cosine. It is equivalent to determining a rational parametrization of the conic $u^2-2\alpha uv+v^2= 1,$ where $\alpha$ is the rational cosine.

The present paper is motivated by the following two problems due to Zolt\'an Bertalan. 
\begin{enumerate}
 \item[(i)] How to choose $x$ and $y$ such that the distances of the clock hands at 2  o'clock and 3 o'clock are integers? 
\item[(ii)] How to choose $x$ and $y$ such that the distances of the clock hands at 2  o'clock and 4 o'clock are integers? 
\end{enumerate}

We generalize and reformulate the above questions as follows. For given $0<\alpha,\beta<\pi$ we are looking for pairs of triangles in which the length of the sides ($z_{\alpha},z_{\beta}$) opposite the angles $\alpha,\beta$ are from some given number field $\mathbb{Q}(\theta)$ and the length of the other two sides ($x,y$) are rational integers. Let $\varphi_1=\cos(\alpha)$ and $\varphi_2=\cos(\beta).$

\begin{center}
\includegraphics[height=3.5cm]{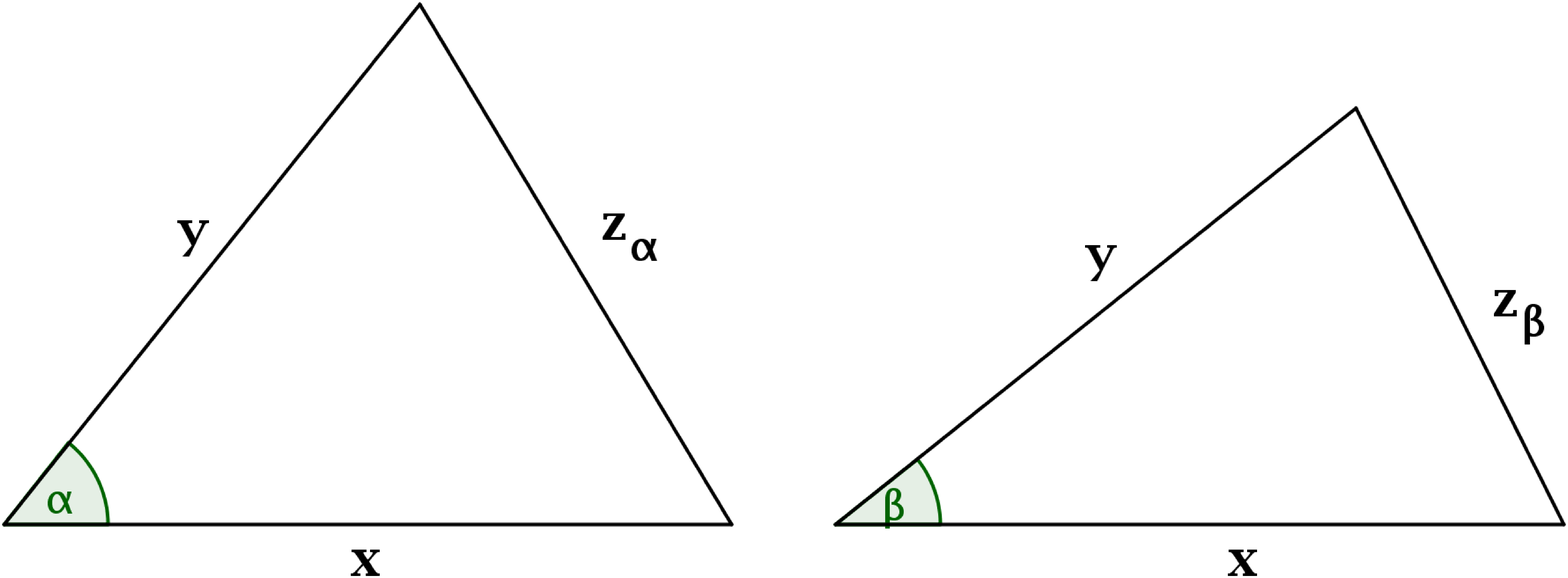}
\end{center}

By means of the law of cosine we obtain the following systems of equations
\begin{eqnarray*}
&& x^2-2\varphi_1xy+y^2=z_{\alpha}^2,\\
&& x^2-2\varphi_2xy+y^2=z_{\beta}^2,
\end{eqnarray*}
After multiplying these equations and dividing by $y^4$ we get
{\small
$$
\mathcal{C}_{\alpha,\beta}: X^4-2(\varphi_1+\varphi_2)X^3+(4\varphi_1\varphi_2+2)X^2-2(\varphi_1+\varphi_2)X+1=Y^2,
$$
}
where $X=x/y$ and $Y=z_{\alpha}z_{\beta}/y^2.$
Suppose $\varphi_1,\varphi_2\in\mathbb{Q}(\theta)$ for some algebraic number $\theta.$ 
Clearly, the hyperelliptic curve $\mathcal{C}_{\alpha,\beta}$ has a rational point $(X,Y)=(0,1),$ so it is isomorphic to an elliptic curve $\mathcal{E}_{\alpha,\beta}$. The rational points of an elliptic curve form a finitely generated group. We are looking for points on $\mathcal{E}_{\alpha,\beta}$ for which the first coordinate of its preimage is rational. If $\mathcal{E}_{\alpha,\beta}$ is defined over $\mathbb{Q}$ and the rank is 0, then there are only finitely many solutions, if the rank is greater than 0, then there are infinitely many solutions. If the elliptic curve $\mathcal{E}_{\alpha,\beta}$ is defined over some number field of degree at least two, then one can apply the so-called elliptic Chabauty method (see \cite{NB1,NB2} and the references given there) to determine all solutions with the required property.

\section{curves defined over $\mathbb{Q}$}

\subsection{$(\alpha,\beta)=(\pi/3,\pi/2)$}
The system of equations in this case is
\begin{eqnarray*}
&&x^2-xy+y^2=z_{\pi/3}^2,\\
&&x^2+y^2=z_{\pi/2}^2.\\ 
\end{eqnarray*}
The related hyperelliptic curve is $\mathcal{C}_{\pi/3,\pi/2}.$
\begin{thm}
There are infinitely many rational points on $\mathcal{C}_{\pi/3,\pi/2}.$
\end{thm}
\begin{proof}
In this case the free rank is 1, as it is given in Cremona's table of elliptic curves \cite{ECTable} (curve 192A1). Therefore there are infinitely many rational points on $\mathcal{C}_{\pi/3,\pi/2}.$
\end{proof}
\begin{cor}
Problem {\rm (i)} has infinitely many solutions.
\end{cor}
Few solutions are given in the following table.
\begin{center}
\begin{table}[h]
\begin{tabular}{|c|c|c|c|}
\hline
$x$ & $y$ & $z_{\pi/3}$ & $z_{\pi/2}$ \\ 
\hline
8 & 15 & 13 & 17 \\ 
\hline
1768 & 2415  & 2993 & 3637 \\
\hline
10130640 & 8109409 & 9286489 & 12976609 \\
\hline
498993199440 & 136318711969 & 517278459169 & 579309170089\\
\hline
\end{tabular}
\end{table}
\end{center}

\subsection{$(\alpha,\beta)=(\pi/2,2\pi/3)$}

The system of equations in this case is
\begin{eqnarray*}
&&x^2+y^2=z_{\pi/2}^2,\\ 
&&x^2+xy+y^2=z_{2\pi/3}^2.
\end{eqnarray*}
The hyperelliptic curve $\mathcal{C}_{\pi/2,2\pi/3}$ is isomorphic to $\mathcal{C}_{\pi/3,\pi/2},$ therefore there are infinitely many rational points on $\mathcal{C}_{\pi/2,2\pi/3}.$

\subsection{$(\alpha,\beta)=(\pi/3,2\pi/3)$}

We have
\begin{eqnarray*}
&&x^2-xy+y^2=z_{\pi/3}^2,\\ 
&&x^2+xy+y^2=z_{2\pi/3}^2.
\end{eqnarray*}
After multiplying these equations we get
\begin{equation}\label{Mor}
x^4+x^2y^2+y^4=(z_{\pi/3}z_{2\pi/3})^2.
\end{equation}
\begin{thm}
If $(x,y)$ is a solution of \eqref{Mor} such that $\gcd(x,y)=1,$ then $xy=0.$
\end{thm}
\begin{proof}
See \cite{Mo} at page 19.
\end{proof}
\begin{cor}
Problem {\rm (ii)} has no solution.
\end{cor}

In the following sections we use the so-called elliptic Chabauty's method (see \cite{NB1},
\cite{NB2}) to determine all points on the curves $\mathcal{C}_{\alpha,\beta}$ for which $X$
is rational. The algorithm is implemented by N. Bruin in MAGMA \cite{MAGMA}, so here
we indicate the main steps only, the actual computations can be carried
out by MAGMA.

\section{curves defined over $\mathbb{Q}(\sqrt{2})$}
\subsection{$(\alpha,\beta)=(\pi/4,\pi/2)$}
The hyperelliptic curve $\mathcal{C}_{\pi/4,\pi/2}$ is isomorphic to
$$
\mathcal{E}_{\pi/4,\pi/2}:\quad v^2=u^3-u^2-3u-1.
$$
The rank of $\mathcal{E}_{\pi/4,\pi/2}$ over $\mathbb{Q}(\sqrt{2})$ is 1, which is less than the degree of $\mathbb{Q}(\sqrt{2}).$ Applying elliptic Chabauty (the procedure "Chabauty" of MAGMA) with
$p=7$, we obtain that $(X,Y)=(0,\pm 1)$ are the only affine points on $\mathcal{C}_{\pi/4,\pi/2}$ with rational first coordinates. Since $X=x/y$ we get that there does not exist appropriate triangles in this case.

\subsection{$(\alpha,\beta)=(\pi/4,\pi/3)$}
The hyperelliptic curve $\mathcal{C}_{\pi/4,\pi/3}$ is isomorphic to
$$
\mathcal{E}_{\pi/4,\pi/3}:\quad v^2=u^3 + (\sqrt{2} - 1)u^2 - 2u - \sqrt{2}.
$$
The rank of $\mathcal{E}_{\pi/4,\pi/2}$ over $\mathbb{Q}(\sqrt{2})$ is 1 and applying elliptic Chabauty's method again with
$p=7$, we obtain that $(X,Y)=(0,\pm 1)$ are the only affine points on $\mathcal{C}_{\pi/4,\pi/3}$ with rational first coordinates. As in the previous case we obtain that there does not exist triangles satisfying the appropriate conditions.

\section{curves defined over $\mathbb{Q}(\sqrt{3})$ and $\mathbb{Q}(\sqrt{5})$}
In the following tables we summarize some details of the computations, that is the pair $(\alpha,\beta),$ the equations of the elliptic curves $\mathcal{E}_{\alpha,\beta},$ the rank of the Mordell-Weil group of these curves over the appropriate number field ($\mathbb{Q}(\sqrt{3})$ or $\mathbb{Q}(\sqrt{5})$), the rational first coordinates of the affine points and the primes we used.
\begin{center}
\begin{tabular}{|c|l|c|c|c|}
\hline
$(\alpha,\beta)$ & \hspace{80pt} $\mathcal{E}_{\alpha,\beta}$ & Rank & $X$ & $p$ \\ 
\hline
$(\pi/6,\pi/2)$ & $v^2 = u^3 - u^2 - 2u$ & 1 & $\{0,\pm 1\}$ & 5 \\ 
$(\pi/6,\pi/3)$ & $v^2 = u^3+(\sqrt{3}-1)u^2-u+(-\sqrt{3}+1)$ & 1 & $\{0\}$ & 7 \\ 
$(\pi/5,\pi/2)$ & $v^2=u^3-u^2+1/2(\sqrt{5}-7)u+1/2(\sqrt{5}-3)$& 1 & $\{0\}$ & 13 \\ 
$(\pi/5,\pi/3)$ & $v^2=u^3+1/2(\sqrt{5}-1)u^2+1/2(\sqrt{5}-5)u-1$& 1 & $\{0, 1\}$ & 13 \\ 
$(\pi/5,2\pi/5)$& $v^2=u^3-2u-1$& 1 & $\{0\}$ & 7\\
$(\pi/5,4\pi/5)$& $v^2=u^3+1/2(-\sqrt{5}+1)u^2-4u+(2\sqrt{5}-2)$& 0 & $\{0\}$ & -\\
\hline
\end{tabular}
\end{center}
In case of $(\alpha,\beta)=(\pi/5,\pi/3)$ we get the following family of triangles given by 
the length of the sides 
$$(x,y,z_{\alpha})=\left(t,t,\frac{-1+\sqrt{5}}{2}t\right) \mbox{ and } (x,y,z_{\beta})=(t,t,t),$$ where $t\in\mathbb{N}.$

\bibliography{all}
\end{document}